\documentclass[11pt]{amsart}

\usepackage{amssymb, amsmath, mathrsfs, mathtools}
\numberwithin{equation}{section}

\usepackage{bm}
\usepackage{commutative-diagrams}
\usepackage[inline]{enumitem}
\usepackage[colorlinks=true, linkcolor=blue, citecolor=blue, urlcolor=blue]{hyperref}

\usepackage{aliascnt, amsthm, cleveref}
\theoremstyle{definition}
\newtheorem{theorem}{Theorem}[section]
\newaliascnt{conjecture}{theorem}
\newtheorem{conjecture}[conjecture]{Conjecture}
\aliascntresetthe{conjecture}
\newaliascnt{definition}{theorem}
\newtheorem{definition}[definition]{Definition}
\aliascntresetthe{definition}
\newaliascnt{example}{theorem}
\newtheorem{example}[example]{Example}
\aliascntresetthe{example}
\newaliascnt{lemma}{theorem}
\newtheorem{lemma}[lemma]{Lemma}
\aliascntresetthe{lemma}
\newaliascnt{remark}{theorem}
\newtheorem{remark}[remark]{Remark}
\aliascntresetthe{remark}
\newaliascnt{proposition}{theorem}
\newtheorem{proposition}[proposition]{Proposition}
\aliascntresetthe{proposition}
\newaliascnt{corollary}{theorem}
\newtheorem{corollary}[corollary]{Corollary}
\aliascntresetthe{corollary}
\Crefname{theorem}{Theorem}{Theorems}
\Crefname{conjecture}{Conjecture}{Conjectures}
\Crefname{definition}{Definition}{Definitions}
\Crefname{example}{Example}{Examples}
\Crefname{lemma}{Lemma}{Lemmas}
\Crefname{proposition}{Proposition}{Propositions}
\Crefname{corollary}{Corollary}{Corollaries}

\counterwithin{equation}{subsection}

\DeclareMathOperator{\dep}{dp}
\DeclareMathOperator{\End}{End}
\DeclareMathOperator{\ldep}{\underline{dp}}
\DeclareMathOperator{\gr}{gr}

\DeclareMathOperator{\Res}{Res}
\DeclareMathOperator{\Specm}{Specm}
\DeclareMathOperator{\spn}{span}
\DeclareMathOperator{\Sym}{Sym}
\DeclareMathOperator{\Tr}{Tr}
\DeclareMathOperator{\Vir}{Vir}

\renewcommand{\hom}{\operatorname{Hom}}

\title[Zhu's algebra and the $C_2$-algebra]{Zhu's algebra and the $C_2$-algebra of a classically free vertex operator algebra}

\author{Bailin Song $^1$}
\email{bailinso@ustc.edu.cn}
\author{Xianlong Zeng $^1$}
\email{zengxl2023@mail.ustc.edu.cn}

\thanks{$^1$ School of Mathematical Sciences, University of Science and Technology of China, Hefei, Anhui 230026, P. R. CHINA}

\begin{document}
	\begin{abstract}
		We prove that, for a classically free vertex algebra, the $C_2$-algebra is naturally isomorphic to the associated graded algebra of Zhu's algebra. 
		As an application, we show that non-trivial holomorphic VOAs of CFT type are not classically free. 
		This includes the affine vertex operator algebra of $E_8$ at level one, as well as the Moonshine VOA.
	\end{abstract}
	
	\maketitle
	
	\section{Introduction}
		Zhu's algebra $A(V)$ and $C_2$-algebra $R_V$ associated with a vertex operator algebra (VOA) $V$ are introduced in Zhu's foundational work \cite{Z}.  
		The Zhu's algebra $A(V)$ plays a key role in the representation theory of $V$: 
			irreducible admissible $V$-modules correspond bijectively to irreducible $A(V)$-modules. 
		This correspondence allows one to translate classification of irreducible admissible $V$-modules into classification of irreducible $A(V)$-modules. 
		The Zhu's algebra has been calculated for many families of VOAs, 
			including lattice VOA $V_L$ for positive-definite even lattice $L$ \cite{DLM1}, 
			affine VOA $L_k(\mathfrak{g})$ for positive integer $k$ \cite{FZ}, and certain $\mathcal{W}$-algebras \cite{A}.
		
		The $C_2$-algebra $R_V$ is closely related to Zhu's algebra. 
		The hypothesis of its finite-dimensionality is assumed in many VOA theories (see for example \cite{DLM3,H2,Z}). 
		There is a conjecture due to \cite{Z}, which is still an open problem.
		\begin{conjecture}\label{rational implies C2-cofinite}
			Let $V$ be a $\mathbb{N}$-graded rational VOA. Then $R_V$ is finite-dimensional.
		\end{conjecture}
		The $C_2$-algebra provides a connection between VOA, Poisson geometry and representation theory via the associated variety $X_V=\Specm R_V$. 
		The 4d/2d corresponding discovered by Beem et al. in \cite{BLLPRR} associates a VOA to a 4-dimensional $\mathcal{N}=2$ superconformal field theory, 
			it is conjectured that the Higgs branch of the 4d theory is isomorphic the associated variety the VOA.
		
		There is a canonical filtration of Zhu's algebra, 
			giving rise to the associated graded algebra $\gr A(V)$ and a surjective morphism $\eta_V:R_V\twoheadrightarrow\gr A(V)$. 
		Then a natural question arises: under what condition is this map bijective?  
		It is proved in \cite{ALY} that $\eta_V$ is an isomorphism when $V$ admits a PBW basis. 
		The question is subtle since $\eta_V$ are isomorphisms for many simple affine vertex algebras $L_k(\mathfrak{g})$ at positive integer level, 
			while it is not an isomorphism for $\mathfrak{g}=E_8$ at level one (see \cite{GG}).
		
		Classical freeness was introduced by van Ekeren and Heluani in the context of chiral homology \cite{EH} 
			and it plays an important role in their computations of chiral homology.
		This property asserts that the associated graded algebra $\gr V$ for Li filtration is the universal differential algebra generated by $R_V$. 
		Several families of VOAs have been shown to be classically free, 
			including Virasoro minimal models $\Vir_{2,q}$ \cite{EH}, affine VOAs $L_1(\mathfrak{sl}_n)$ \cite{Ka}
			and $L_k(\mathfrak{sp}_{2n})$ for positive integer $k$ \cite{LS}.
		
		In this work, we prove the following theorem (cf. \Cref{main}).
		\begin{theorem}\label{thm:main}
			If an $\mathbb{N}$-graded vertex algebra $V$ is classically free, then the natural projection
			$$\eta_V:R_V\twoheadrightarrow\gr A(V)$$
			is an isomorphism.
		\end{theorem}
		As an immediate consequence, 
			the \Cref{rational implies C2-cofinite} is true for classically free VOAs, since $\dim A(V)<\infty$ if $V$ is rational (see \Cref{Zhu theory}).
		\begin{theorem}
			Let $V$ be an $\mathbb{N}$-graded classically free VOA. Then $R_V$ is finite-dimensional if $V$ is rational.
		\end{theorem}
		Conversely, if $\dim R_V\neq\dim A(V)$, $V$ is not classically free. And we have the following theorem (cf. \Cref{Zhu=C}).
		\begin{theorem}
			A non-trivial holomorphic VOA of CFT type is not classically free.
		\end{theorem}
		This includes the lattice VOAs of positive-definite even self-dual lattices and the Moonshine VOA (cf. \Cref{E8,Leech,moonshine}). 
		In particular, the affine VOA $L_1(E_8)$, which is the lattice VOA of lattice $E_8$, is not classically free. 
		This is confirmed by a direct calculation in \Cref{sec6}.
		It provides a counterexample to a well accepted conjecture that every simple affine VOA at positive integer level is classically free \cite{F}.
		
		The key step in the proof of \Cref{thm:main} is that we introduce a filtration $O_n$ of $O(V)=V\circ V$. The analogue between $O_n$ and $F^1V\cap V_{\leq n}$ is a bridge connecting $O(V)$ and $\gr V$
			
		We also use a family of auxiliary operations $\bullet_n$, 
			which are deformations of the mode products $a_{(n)}b$ and coincides with the products $*$ and $\circ$ defining Zhu's algebras at $n=-1,-2$.
		By systematically studying the interaction between these operations and the filtration $O_n$, 
			we establish the correspondence between relations in $F^1V\cap V_n$ and relations in $O_n/O_{n-1}$, under the assumption of classical freeness.
		
		This paper is organized as follows. 
		We recall the basics of vertex algebras and related structures in \Cref{sec2}, 
			and introduce the operations $\bullet_n$ and establish their basic properties in \Cref{sec3}.
		In \Cref{sec4}, we completes the proof of \Cref{thm:main}.
		In \Cref{sec5}, we lists several consequences of \Cref{thm:main}, including that $L_1(E_8)$ and the Moonshine VOA are not classically free. 
		In \Cref{sec6}, we prove that $L_1(E_8)$ is not classically free by a direct calculation.
		
	\section{Preliminaries}\label{sec2}
		In this section we recall the basic definitions and well-known facts that we will use in the sequel.
		\subsection{Graded vertex algebras}
			We follow the definition of vertex algebras in \cite[Proposition 4.8]{K2}, with a slight simplification regarding parity: 
			In this paper, we use vertex algebra to refer specifically to a pure even space, and use vertex superalgebra for a general superspace (if needed).
			
			\begin{definition}
				A \emph{vertex algebra} consists of a vector space $V$, a distinguished vector $\bm1\in V$, and a linear map
				\begin{align*}
					Y(-,z):V&\to(\End V)[[z,z^{-1}]] \\
					a&\mapsto Y(a,z)=\sum_{n\in\mathbb{Z}}a_{(n)}z^{-n-1}
				\end{align*}
				such that the following axioms hold for all $a,b\in V$:
				\begin{enumerate}[label=(\roman*)]
					\item\emph{field}:
						$$a_{(n)}b=0\quad\text{for }n\gg0;$$
					\item\emph{vacuum}:
						\begin{align*}
							Y(\bm1,z)&=1, \\
							Y(a,z)\bm1|_{z=0}&=a;
						\end{align*}
					\item\emph{Borcherds identity}: 
						Let $F(z,w)$ be a rational function in $z$ and $w$ with poles only at $z=0,w=0$ or $z=w$. Then
						\begin{equation}\label{Borcherds}\begin{aligned}
							\Res_{z-w}&Y(Y(a,z-w)b,w)\iota_{w,z-w}F(z,w) \\
							&=\Res_z\left(Y(a,z)Y(b,w)\iota_{z,w}-Y(b,w)Y(a,z)\iota_{w,z}\right)F(z,w).
						\end{aligned}\end{equation}\end{enumerate}\end{definition}
			
			As special cases of the Borcherds identity \eqref{Borcherds}, we have the following commutative and associative relations:
			\begin{align}
			\label{commutator}
				[a_{(m)},b_{(n)}]&=\sum_{j=0}^\infty\binom{m}{j}(a_{(j)}b)_{(m+n-j)}, \\
			\label{associator}
				(a_{(m)}b)_{(n)}&=\sum_{j=0}^\infty(-1)^j\binom{m}{j}\left(a_{(m-j)}b_{(n+j)}-(-1)^mb_{(m+n-j)}a_{(j)}\right).
			\end{align}
			We also use \eqref{associator} in its series form:
			\begin{equation}\label{series}\begin{aligned}
				Y(a_{(m)}b,w)=\Res_z\Big(&Y(a,z)Y(b,w)\iota_{z,w}(z-w)^m \\
				&-Y(b,w)Y(a,z)\iota_{w,z}(z-w)^m\Big).
			\end{aligned}\end{equation}
			
			\begin{definition}
				Let $V$ be a vertex algebra. Then its \emph{infinitesimal translation operator} $T\in\End V$ is defined by
				\begin{equation}\label{translation}
					Ta=a_{(-2)}\bm1.
			\end{equation}\end{definition}
			
			The infinitesimal translation operator has the following basic properties:
			\begin{gather}
			\label{covariance}
				[T,Y(a,z)]=\frac\partial{\partial z}Y(a,z)=Y(Ta,z), \\
			\label{generating}
				a_{(-1-n)}b=\frac1{n!}(T^na)_{(-1)}b, \\
			\label{derivation}
				T(a_{(k)}b)=(Ta)_{(k)}b+a_{(k)}Tb,
			\end{gather}
			for all $a,b\in V$, $k\in\mathbb{Z}$ and $n\in\mathbb{N}$.
			
			\begin{definition}
				A \emph{graded vertex algebra} consists of a vertex algebra $V$ and a diagonalizable operator $H\in\End V$, such that for all $a\in V$
				\begin{equation}\label{hamiltonian}
					[H,Y(a,z)]=z\frac\partial{\partial z}Y(a,z)+Y(Ha,z).
				\end{equation}\end{definition}
			
			We denote the eigenspace decomposition of $V$ with respect to $H$ by
				$$V=\bigoplus_{n\in\mathbb{C}}V_n,\quad V_n=\ker(H-nI).$$
			An eigenvector $a$ of $H$ is called \emph{homogeneous}, and its eigenvalue is denoted by $\deg a$. For a graded vertex algebra $V$ we have
			$$\deg\bm1=0$$
			and
			\begin{align}
			\label{grading}
				(V_m)_{(k)}(V_n)&\subseteq V_{m+n-k-1}, \\
			\label{shifting}
				TV_n&\subseteq V_{n+1},
			\end{align}
			for all $m,n\in\mathbb{N}$ and $k\in\mathbb{Z}$.
			
			This paper focus on $\mathbb{N}$-graded vertex algebras, i.e. we assume that
			$$V=\bigoplus_{n\in\mathbb{N}}V_n$$
			unless otherwise stated.
			
			We finally introduce some concepts that will be used in \Cref{sec5}.
			\begin{definition}
				A \emph{vertex operator algebra} consists of a graded vertex algebra $V$ and a distinguished element $\omega\in V_2$ with
				$$Y(\omega,z)=\sum_{n\in\mathbb{Z}}L_nz^{-2-n},$$
				such that
				$$L_{-1}=T,\quad L_0=H,$$
				and
				$$[L_m,L_n]=(m-n)L_{m+n}+\delta_{m+n,0}\frac{m^3-m}{12}cI,$$
				where $c$ is a constant.
			\end{definition}
			
			Since this paper does not involve the specific structures of the modules over VOAs, we only briefly describe the definition of rationality. 
			See \cite{DJ} for details.
			\begin{definition}
				A vertex operator algebra $V$ is called \emph{rational}, if the category of admissible $V$-modules is semisimple, 
					i.e. any admissible $V$-module is a direct sum of simple admissible $V$-modules.
			\end{definition}
			
			\begin{definition}
				A vertex operator algebra $V$ is called \emph{holomorphic}, 
					if $V$ is rational, and $V$ is the unique simple admissible $V$-module, up to an isomorphism.
			\end{definition}
			
			\begin{definition}
				A vertex operator algebra $V$ is called of \emph{CFT type}, provided it is $\mathbb{N}$-graded and $V_0=\mathbb{C}\bm1$.
			\end{definition}
			
		\subsection{Zhu's algebra and its associated graded algebra}
			Zhu \cite{Z} associates with each graded vertex algebra $V$ an associative algebra $A(V)$, defined as follows.
			\begin{definition}
				Define two bilinear operations $*$ and $\circ$ on $V$ as follows. For $a$ homogeneous and $b\in V$, let
				\begin{align}
					\label{star}
						a*b=\Res_z\frac{(1+z)^{\deg a}}zY(a,z)b, \\
					\label{circle}
						a\circ b=\Res_z\frac{(1+z)^{\deg a}}{z^2}Y(a,z)b,
				\end{align}
				and extend them to $V\times V$ bilinearly. Let $A(V)=V/O(V)$, where
				$$O(V)=\spn\{a\circ b\mid a,b\in V\}.$$
				and denote by $[a]$ the image of $a$ under the natural projection $V\twoheadrightarrow A(V)$.
			\end{definition}
			
			The following theorem is due to \cite{Z} and \cite{DLM2}:
			\begin{theorem}\label{Zhu theory}
				Let $V$ be a vertex operator algebra. Then
				\begin{enumerate}[label=(\arabic*)]
					\item	$ A(V)$ is an associative algebra with unit $[\bm1]$, under the multiplication induced by bilinear operation $*$.
					\item	Let $M=\bigoplus\limits_{n=0}^\infty M_n$ be an admissible $V$-module with $M_0\neq0$. Then the linear map
						\begin{align*}
							o:V&\to\End M_0 \\
							a&\mapsto a_{(\deg a-1)}
						\end{align*}
						defines an $A(V)$-module structure on $M_0$.
					\item	The map $M\mapsto M_0$ gives a bijection 
							from the set of equivalent class of irreducible admissible $V$-modules, to the set of equivalent class of irreducible $A(V)$-modules.
					\item	If $V$ is rational, then $ A(V)$ is finite-dimensional and semisimple.
			\end{enumerate}\end{theorem}
			
			\begin{remark}
				Note that the Virasoro element $\omega$ is not involved in the construction of Zhu's algebra. 
				In fact, the concept of Zhu's algebra can be generalized to $\mathbb{N}$-graded vertex algebras. 
				\Cref{grading-1,grading-2,Jacobi} are sufficient to show that $A(V)$ is an associative algebra. 
				Furthermore, \cite{BK} shows that 
					the representation-theoretic properties, i.e. \Cref{Zhu theory}(ii)(iii), also hold for $\mathbb{N}$-graded vertex algebras.
			\end{remark}
			
			We also consider the associated graded algebra $\gr A(V)$. Let
			$$V_{\leq n}=\bigoplus_{j\leq n}V_j$$
			and denote by $A_n(V)$ the image of $V_{\leq n}$ under the projection $V\twoheadrightarrow A(V)$. 
			It follows from \eqref{grading} that $A_n(V)$ defines a nice increasing filtration of $A(V)$ in the following sense:
			$$A_m(V)* A_n(V)\subseteq A_{m+n}(V).$$
			Therefore, the associated graded algebra
			$$\gr A(V)=\bigoplus_{n\in\mathbb{N}} A_n(V)/A_{n-1}(V)$$
			is a graded associative algebra with respect to the multiplication induced by $*$. 
			Moreover, the following lemma \cite[Lemma 2.1.3]{Z} implies that $\gr A(V)$ is commutative.
			\begin{lemma}
				Let $V$ be an $\mathbb{N}$-graded VOA. Then for $a\in V$ homogeneous and $b\in V$, we have
				$$a*b-b*a\equiv\Res_z(z+1)^{\deg a-1}Y(a,z)b\pmod{O(V)}.$$
				As an immediate consequence, for $a,b\in V$ homogeneous, we have
				\begin{equation*}
					a*b-b*a\in V_{\leq\deg a+\deg b-1}+O(V).
			\end{equation*}\end{lemma}
			
		\subsection{Li filtration and its associated graded algebra}
			Each vertex algebra is canonically decreasingly filtered \cite{L}, as follows.
			
			\begin{definition}
				For a vertex algebra $V$, let
				\begin{equation}\label{filtration}
					F^pV=\spn\left\{a^1_{(-n_1-1)}\cdots a^r_{(-n_r-1)}\bm1\middle|n_i\geq0,n_1+\cdots+n_r\geq p\right\}.
				\end{equation}
				Then it is obvious that
				$$V=F^0V\supseteq F^1V\supseteq F^2V\supseteq\cdots.$$
				Let
				\begin{align*}
					R_V&=V/F^1V, \\
					\gr V&=\bigoplus_{p\in\mathbb{N}}F^pV/F^{p+1}V.
			\end{align*}\end{definition}
			
			The following theorem is due to \cite[Lemmas 2.10, 2.14, 4.1, 4.2]{L}:
			\begin{theorem}\label{grV}
				\begin{enumerate}[label=(\arabic*)]
					\item	If we take $F^pV=F^0V=V$ for $p<0$, then
						$$(F^pV)_{(k)}(F^qV)\subseteq F^{p+q-k-1}V\quad\text{for}\quad k\in\mathbb{Z}$$
						and
						$$(F^pV)_{(k)}(F^qV)\subseteq F^{p+q-k}V\quad\text{for}\quad k\in\mathbb{N}.$$
					\item	Let $(\gr V)_p=F^pV/F^{p+1}V$. Then $(a,b)\mapsto a_{(-1)}b$ induces a multiplication
						$$(\gr V)_p\times(\gr V)_q\to(\gr V)_{p+q}.$$
						This multiplication on $\gr V$ is commutative and associative.
					\item	$\gr V$ is a differential algebra generated by $R_V=(\gr V)_0$, under the multiplication defined in (2) and the derivation defined by
						\begin{align*}
							\partial:(\gr V)_p&\to(\gr V)_{p+1} \\
							a+F^{p+1}V&\mapsto Ta+F^{p+2}V.
						\end{align*}\end{enumerate}\end{theorem}
			
			$R_V$ is related to Zhu's algebra via the following theorem \cite[Proposition 3.3]{ALY}.
			\begin{theorem}\label{epimorphism}
				The	following map is a well-defined surjective homomorphism of commutative algebras:
				\begin{align*}
					\eta_V:R_V&\to\gr A(V)\\
					a+F^1V&\mapsto a+V_{<\deg a}+O(V).
			\end{align*}\end{theorem}
			
		\subsection{Classical freeness}
			In view of \Cref{grV}, it is natural to ask to what extent $\gr V$ is determined by $R_V$. 
			The simplest case, in which $\gr V$ is completely determined by $R_V$, is formally introduced in \cite{EH}.
			
			\begin{definition}
				A \emph{differential algebra} is a unital commutative $\mathbb{C}$-algebra $R$, together with a derivation $\partial:R\to R$.
				
				A map $f:(R,\partial)\to(R',\partial')$ is called a \emph{differential algebra homomorphism}, provided $f$ is an algebra homomorphism and
				$$\partial'f(x)=f(\partial x),\quad\forall x\in R.$$
			\end{definition}
			
			\begin{proposition}\label{jet algebra}
				For any unital commutative (not necessarily finite generated) $\mathbb{C}$-algebra $R$, 
					there exists a unique (up to an isomorphism) differential algebra $\mathcal{J}_\infty R$, 
					together with an algebra homomorphism $j:R\to\mathcal{J}_\infty R$, such that
				\begin{equation}\label{universal}
					j^*:\hom_{\mathbf{Diff.Alg}}(\mathcal{J}_\infty R,A)\to\hom_{\mathbf{Alg}}(R,A)
				\end{equation}
				is a bijection for any differential algebra $A$. The differential algebra $\mathcal{J}_\infty R$ is called the \emph{jet algebra} of $R$.
				\begin{proof}
					The uniqueness is due to Yoneda's lemma, hence it suffices to prove the existence. 
					Throughout the whole proof, we use $(A,\partial_A)$ to denote an arbitrary differential algebra.
					
					First consider the polynomial ring $\mathbb{C}[S]$, where $S$ is the set of indeterminates. 
					Let $JS=\{\partial^nx|n\geq0,x\in S\}$ and equip $\mathcal{J}_\infty\mathbb{C}[S]=\mathbb{C}[JS]$ with the derivation
					$$\partial:\partial^nx\mapsto\partial^{n+1}x,\quad n\geq0, x\in S.$$
					Then we have the natural inclusion
					\begin{align*}
						j:\mathbb{C}[S]&\hookrightarrow\mathcal{J}_\infty\mathbb{C}[S] \\
						x&\mapsto\partial^0x.
					\end{align*}
					Now we prove the condition \eqref{universal}. 
					Since $\mathcal{J}_\infty\mathbb{C}[S]$ is generated by $j(\mathbb{C}[S])$ as a differential algebra, $j^*$ is injective. 
					Conversely, let $\varphi:\mathbb{C}[S]\to A$ is an algebra homomorphism. Then
					\begin{align*}
						J_\varphi:\mathcal{J}_\infty\mathbb{C}[S]&\to A \\
						\partial^nx&\mapsto\partial_A^n\varphi(x)
					\end{align*}
					is obviously a well-defined differential algebra homomorphism and $j^*J_\varphi=\varphi$.
					
					Next let $R$ be a general algebra. Choose a generating set $S$, we can write
					$$R=\mathbb{C}[S]/I.$$
					Let $\mathcal{J}_\infty I=\langle\partial^nf|n\geq0,f\in I\rangle\subseteq\mathbb{C}[JS]$, 
					then $\mathcal{J}_\infty R=\mathbb{C}[JS]/\mathcal{J}_\infty I$ is a differential algebra, since $\mathcal{J}_\infty I$ is a differential ideal. 
					And we have the following commutative diagram
					\begin{center}\begin{codi}
						\obj {	|(S)|\mathbb{C}[S]		 && R \\
									|(JS)|\mathbb{C}[JS]	&& |(JR)|\mathcal{J}_\infty R \\};
						\mor S \pi :->> R \overline{j} :-> JR;
						\mor * j :-> JS \mathcal{J}_\infty\pi :->> *;
					\end{codi}\end{center}
					where $\mathcal{J}_\infty\pi$ is the natural projection. Now we prove the condition \eqref{universal}.
					\begin{center}\begin{codi}
						\obj {	|(S)|\hom_{\mathbf{Alg}}(\mathbb{C}[S],A)				&&& |(R)|\hom_{\mathbf{Alg}}(R,A) \\
									|(JS)|\hom_{\mathbf{Diff.Alg}}(\mathbb{C}[JS],A)	&&& |(JR)|\hom_{\mathbf{Diff.Alg}}(\mathcal{J}_\infty R,A) \\};
						\mor JR \overline{j}^* :-> R \pi^* :-> S;
						\mor * (\mathcal{J}_\infty\pi)^* :-> JS j^* :-> *;
					\end{codi}\end{center}
					
					In the diagram above, $j^*$ is injective as we have shown, and $(\mathcal{J}_\infty\pi)^*$ is injective since $\mathcal{J}_\infty\pi$ is surjective. 
					Hence
					$$\pi^*\circ\overline{j}^*=j^*\circ(\mathcal{J}_\infty\pi)^*$$
					is injective, which implies that $\overline{j}^*$ is injective.
					
					On the other hand, let $\varphi:R\to A$ is an algebra homomorphism. 
					Then the homomorphism $\pi^*\varphi:\mathbb{C}[S]\to A$ 
						is uniquely extended to a differential algebra homomorphism $J_{\pi^*\varphi}:\mathbb{C}[JS]\to A$.
					For any $f\in I$, we have
					$$J_{\pi^*\varphi}(\partial^nf)=\partial_A^n\varphi(\pi(f))=0.$$
					Thus $J_{\pi^*\varphi}$ induces a differential algebra homomorphism
					$$J_\varphi:\mathcal{J}_\infty R\to A.$$
					Now a routine diagram chasing shows that
					$$\pi^*\overline{j}^*J_\varphi=j^*(\mathcal{J}_\infty\pi)^*J_\varphi=j^*J_{\pi^*\varphi}=\pi^*\varphi.$$
					It follows that $\overline{j}^*J_\varphi=\varphi$.
			\end{proof}\end{proposition}
			
			By definition of jet algebra, for any differential algebra $A$ generated by $R\subseteq A$, there is a canonical projection
			\begin{align*}
				\pi:\mathcal{J}_\infty R&\twoheadrightarrow A \\
				x&\mapsto x,\quad x\in R.
			\end{align*}
			\begin{definition}
				A vertex algebra $V$ is called \emph{classically free} if the canonical projection
				$$\pi:\mathcal{J}_\infty R_V\twoheadrightarrow\gr V$$
				is an isomorphism.
			\end{definition}
			
			Let $V$ be strongly generated by $S\subseteq V$ (not necessarily finite). Then by \Cref{grV}, we have the natural projection of differential algebras
			\begin{equation}\label{natural projection}\begin{aligned}
				\varphi:\mathcal{J}_\infty\mathbb{C}[S]&\twoheadrightarrow\gr V \\
				\partial^nx&\mapsto T^nx\in(\gr V)_n\quad\text{for }x\in S
			\end{aligned}\end{equation}
			The surjective homomorphism $\varphi|_{\mathbb{C}[S]}:\mathbb{C}[S]\to R_V$ gives
			$$R_V\cong\mathbb{C}[S]/I$$
			Then according to construction in \Cref{jet algebra}, $V$ is classically free iff
			$$\gr V\cong\mathcal{J}_\infty\mathbb{C}[S]/\langle\partial^nf|n\geq0,f\in I\rangle$$
			This description is independent to the choice of strong generating set, including whether the strong generating set is finite or not. 
			This is important since we need a large generating set in \Cref{sec4}.
			
	\section{Operations \texorpdfstring{$\bullet_n$}{} and their properties}\label{sec3}
		In this section we define a family of auxiliary operations $\bullet_n$ and establish several of their basic properties. 
		The definition is inspired by operations \eqref{star}\eqref{circle} and has already appeared in \cite{BK}. 
		These properties suggest that, for $n<0$, the operations $\bullet_n$ behave analogously to the mode products $(a,b)\mapsto a_{(n)}b$.
		
		In this section, $V$ denote a $\mathbb{N}$-graded vertex algebra.
		
		\subsection{Definitions}
			\begin{definition}
				For each $n\in\mathbb{Z}$, we define a bilinear operation $\bullet_n$ on $V$ as follows. For $a$ homogeneous and $b\in V$, let
				$$a\bullet_nb=\Res_zz^n(1+z)^{\deg a}Y(a,z)b.$$
				In particular, we have $a*b=a\bullet_{-1}b$ and $a\circ b=a\bullet_{-2}b$.
				
				Analogously to \eqref{translation}, for $a\in V$ we define
				\begin{equation*}
					\widetilde{T}a=a\circ\bm1=(T+H)a.
				\end{equation*}\end{definition}
			
			\begin{lemma}\label{expansion}
				For every $a\in V$ homogeneous and $b\in V$, we have
				$$a\bullet_nb=\sum_{i=0}^{\deg a}\binom{\deg a}ia_{(n+i)}b.$$
				In particular, we have
				$$a\bullet_nb\equiv a_{(n)}b\pmod{V_{<\deg a+\deg b-n-1}}$$
				if $b$ is also homogeneous.
				\begin{proof}
					By expanding $(1+z)^{\deg a}$ we obtain
					\begin{equation*}
						a\bullet_nb=\sum_{i=0}^{\deg a}\binom{\deg a}i\Res_zz^{n+i}Y(a,z)b=\sum_{i=0}^{\deg a}\binom{\deg a}ia_{(n+i)}b.
						\qedhere\end{equation*}\end{proof}\end{lemma}
			
			\begin{lemma}\label{pseudo-covariance}
				For every $a,b\in V$ and $n\in\mathbb{Z}$, we have
				\begin{equation}\label{pseudo-covariance-eq}
					(\widetilde{T}a)\bullet_nb=-(n+1)a\bullet_nb-na\bullet_{n-1}b
				\end{equation}
				\begin{proof}
					Since \eqref{pseudo-covariance-eq} is linear in $a$, we can assume that $a$ is homogeneous. 
					$\deg Ta=\deg a+1$ due to \eqref{covariance}. Then applying \eqref{covariance} yields
					\begin{align*}
						Ta\bullet_nb&=\Res_zz^n(1+z)^{\deg a+1}\frac\partial{\partial z}Y(a,z)b \\
						&=-\Res_z\frac\partial{\partial z}\left(z^n(1+z)^{\deg a+1}\right)Y(a,z)b \\
						&=-\Res_z\left((n+\deg a+1)z^n+nz^{n-1}\right)(1+z)^{\deg a}Y(a,z)b \\
						&=-(n+\deg a+1)a\bullet_nb-na\bullet_{n-1}b.
						\qedhere\end{align*}\end{proof}\end{lemma}
			
			\begin{definition}
				Let $O_n=\spn\{a\circ b\mid\deg a+\deg b<n\}\subseteq O(V).$
			\end{definition}
			
			The increasing filtration $O_n$ of $O(V)$ is an analogue to the increasing filtration $F^1V\cap V_{\leq n}$ of $F^1V$, noting that
			$$F^1V\cap V_{\leq n}=\spn\{a_{(-2)}b|\deg a+\deg b<n\}.$$
			We use this filtration to study $\gr A(V)$, by considering $V_{\leq n}/O_{n-1}$.
			
		\subsection{Basic properties of \texorpdfstring{$\bullet_n$}{}}
			\begin{proposition}\label{grading-2}
				For $k\leq-2$, we have $V_m\bullet_kV_n\subseteq O_{m+n-k-1}$.
				\begin{proof}
					We prove the proposition by decreasing induction on $k$. The base case $k=-2$ is immediate from the definition.
					Now, let $a\in V_m$ and $b\in V_n$. Then by \Cref{pseudo-covariance}, we have
					\begin{align*}
						a\bullet_{k-1}b&=-\frac1kTa\bullet_kb-\frac1k(k+\deg a+1)a\bullet_kb \\
						&\in O_{(m+1)+n-k-1}+O_{m+n-k-1} \\
						&=O_{m+n-(k-1)-1}.
					\qedhere\end{align*}\end{proof}\end{proposition}
			
			\begin{proposition}\label{Jacobi}
				Let $a,b,c\in V$ be homogeneous and
				$$d=\deg a_{(m)}b_{(n)}c-1=\deg a+\deg b+\deg c-m-n-3.$$
				Then the following identities hold:
				\begin{enumerate}[label=(\roman*)]
					\item	When $m+n\leq-3$, we have
						\begin{equation}\label{commutator-eq}
							a\bullet_m(b\bullet_nc)-b\bullet_n(a\bullet_mc)\equiv\sum_{j=0}^\infty\binom{m}{j}(a_{(j)}b)\bullet_{m+n-j}c\pmod{O_d};
						\end{equation}
						When $m=n=-1$, we have
						\begin{equation}\label{commutator-1-eq}\begin{aligned}
							a*(b*c)-b*(a*c)\equiv\sum_{j=0}^\infty\binom{-1}{j}&(a_{(j)}b)\bullet_{-2-j}c \\
							&\pmod{O_d+V_{<\deg a+\deg b}*c};
						\end{aligned}\end{equation}
					\item	When $m,n<0$, we have
						\begin{equation}\label{associator-eq}\begin{aligned}
							(a\bullet_mb)\bullet_nc\equiv\sum_{j=0}^\infty(-1)^j&\binom{m}{j}\Big(a\bullet_{m-j}(b\bullet_{n+j}c) \\
							&-(-1)^mb\bullet_{m+n-j}(a\bullet_jc)\Big)\pmod{O_d}.
						\end{aligned}\end{equation}
				\end{enumerate}
			\begin{proof}\begin{enumerate}[label=(\roman*)]
				\item	By definition we have
					\begin{align}
					\notag
						&a\bullet_m(b\bullet_nc)-b\bullet_n(a\bullet_mc) \\
					\label{def of commutator-eq}
						=&\Res_z\Res_wz^m(1+z)^{\deg a}w^n(1+w)^{\deg b}[Y(a,z),Y(b,w)]c.
					\end{align}
					Applying Borcherds identity \eqref{Borcherds} to \eqref{def of commutator-eq}, we obtain

					\begin{align*}
						&a\bullet_m(b\bullet_nc)-b\bullet_n(a\bullet_mc) \\
						=&\Res_w\Res_{z-w}z^m(1+z)^{\deg a}w^n(1+w)^{\deg b}Y(Y(a,z-w)b,w)c \\
						=&\sum_{j=0}^\infty\sum_{i=0}^{\deg a}\binom{\deg a}i\binom mj\Res_ww^{m+n-j}(1+w)^{\deg a+\deg b-i} \\
						&\qquad\cdot\Res_{z-w}(z-w)^{i+j}Y(Y(a,z-w)b,w)c \\
						=&\sum_{j=0}^\infty\sum_{i=0}^{\deg a}\binom{\deg a}i\binom mj\Res_ww^{m+n-j}(1+w)^{j+1} \\
						&\qquad\cdot(1+w)^{\deg a+\deg b-i-j-1}Y(a_{(i+j)}b,w)c \\
						=&\sum_{j=0}^\infty\sum_{i=0}^{\deg a}\sum_{k=0}^{j+1}\binom{\deg a}i\binom mj\binom{j+1} k(a_{(i+j)}b)\bullet_{m+n-k+1}c.
					\end{align*}
					By \Cref{grading-2}, the terms with $m+n-k+1\leq-2$ and $i+j\geq k$ lie in $O_d$ and thus be omitted. 
					We now check the remaining terms:
					\begin{itemize}
						\item	When $m+n\leq-3$: the remaining terms correspond to $i=0$ and $k=j+1$, and \eqref{commutator-eq} follows.
						\item	When $m=n=-1$: the remaining terms correspond to $i=0$ and $k=j+1$, together with the terms with $k=0$. 
							Note that $a_{(i+j)}b\in V_{<\deg a+\deg b}$, the sum of terms with $k=0$ is
							$$\sum_{j=0}^\infty\sum_{i=0}^{\deg a}\binom{\deg a}i\binom{-1}j(a_{(i+j)}b)*c\in V_{<\deg a+\deg b}*c.$$
							Hence \eqref{commutator-1-eq} follows modulo $V_{<\deg a+\deg b}*c$ additionally.
					\end{itemize}
				\item	By \Cref{expansion}, we have
					\begin{align}
					\notag
						&(a\bullet_mb)\bullet_nc \\
					\notag
						=&\Res_w\sum_{i=0}^{\deg a}\binom{\deg a}iw^n(1+w)^{\deg a+\deg b-m-i-1}Y(a_{(m+i)}b,w)c \\
					\notag
						=&\sum_{i=0}^{\deg a}\binom{\deg a}i\Res_w\Big(w^n(1+w)^{\deg b-m-1} \\
					\notag
						&\qquad\cdot(1+w)^{\deg a-i}\Res_{z-w}(z-w)^{m+i}Y(Y(a,z-w)b,w)c\Big) \\
					\label{def of associator-eq}
						=&\Res_w\Res_{z-w}\Big((1+z)^{\deg a}w^n(1+w)^{\deg b-m-1}(z-w)^m \\
					\notag
						&\hspace{16em}\cdot Y(Y(a,z-w)b,w)c\Big).
					\end{align}
					As in part (i), applying Borcherds identity \eqref{Borcherds} to \eqref{def of associator-eq} gives
					\begin{align*}
						&(a\bullet_mb)\bullet_nc \\
						=&\sum_{j=0}^\infty\Res_z\Res_w(-1)^j\binom mjz^{m-j}w^{n+j}(1+w)^{-m-1} \\
						&\qquad\cdot(1+z)^{\deg a}(1+w)^{\deg b}Y(a,z)Y(b,w)c \\
						&\quad-(-1)^m\sum_{j=0}^\infty\Res_w\Res_z(-1)^j\binom mjz^jw^{m+n-j}(1+w)^{-m-1} \\
						&\qquad\cdot(1+z)^{\deg a}(1+w)^{\deg b}Y(b,w)Y(a,z)c \\
						=&\sum_{k=0}^{-m-1}\binom{-m-1}k\sum_{j=0}^\infty(-1)^j\binom mj \\
						&\qquad\cdot\left(a\bullet_{m-j}(b\bullet_{n+j+k}c)-(-1)^mb\bullet_{m+n+k-j}(a\bullet_jc)\right).
					\end{align*}
					Since $m+k\leq-1$ and $n<0$, we have $m+n+k-j\leq-2$. On the other hand, if $k>0$, then $m\leq-2$ and thus $m-j\leq-2$.
					
					Now by \Cref{grading-2}, the terms with $k>0$ lie in $O_d$. 
					The remaining terms, which correspond to $k=0$, are exactly the right-hand side of \eqref{associator-eq}.
				\qedhere\end{enumerate}\end{proof}\end{proposition}
			
			\begin{proposition}\label{Zhu-derivation}\begin{enumerate}[label=(\roman*)]
				\item	For $a,b\in V$ homogeneous and $n\geq0$, we have
					$$a\bullet_{-1-n}b\equiv\frac1{n!}(\widetilde{T}^na)*b\pmod{O_{\deg a+\deg b+n-1}}.$$
				\item	For every $a,b\in V$ and $n\in\mathbb{Z}$, we have
					\begin{equation}\label{derivation-eq}
						\widetilde{T}(a\bullet_nb)=(\widetilde{T}a)\bullet_nb+a\bullet_n\widetilde{T}b.
					\end{equation}\end{enumerate}
			\begin{proof}\begin{enumerate}[label=(\roman*)]
				\item	The case $n=0$ is the identity $a*b\equiv a*b$. And by \Cref{pseudo-covariance},
					$$(\widetilde{T}a)\bullet_{-1-n}b=na\bullet_{-1-n}b+(n+1)a\bullet_{-2-n}b.$$
					Applying \Cref{grading-2}, we obtain
					$$a\bullet_{-2-n}b\equiv\frac1{n+1}(\widetilde{T}a)\bullet_{-1-n}b\pmod{O_{\deg a+\deg b+n}}.$$
					Now the induction hypothesis applies.
				\item	Since \eqref{derivation-eq} is linear in $a$, we can assume the $a$ is homogeneous. 
					By \eqref{covariance} and \eqref{hamiltonian}, we have
					$$[T+H,Y(a,z)]=(1+z)Y(Ta,z)+Y(Ha,z).$$
					It follows that
					\begin{align*}
						&\widetilde{T}(a\bullet_nb)-a\bullet_n\widetilde{T}b \\
						=&\Res_zz^n(1+z)^{\deg a}[T+H,Y(a,z)]b \\
						=&\Res_zz^n(1+z)^{\deg a+1}Y(Ta,z)b+\Res_zz^n(1+z)^{\deg a}Y(Ha,z)b \\
						=&(\widetilde{T}a)\bullet_nb.
					\qedhere\end{align*}\end{enumerate}\end{proof}\end{proposition}
			
			\begin{proposition}\label{grading-1}
				$V_m*O_n\subseteq O_{m+n}$ and $O_n*V_m\subseteq O_{m+n}$ for $m,n\in\mathbb{Z}$.
				\begin{proof}
					Let $\deg a=m$ and $\deg b+\deg c<n$. Then by \Cref{Jacobi},
					\begin{align*}
						a*(b\circ c)&\equiv b\circ(a*c)+\sum_{j=0}^\infty\binom{-1}j(a_{(j)}b)\bullet_{-3-j}c\pmod{O_{m+n-1}}, \\
						(b\circ c)*a&\equiv\sum_{j=0}^\infty(-1)^j\binom{-2}j\left(b\bullet_{-2-j}(c\bullet_{j-1}a)-c\bullet_{-3-j}(b\bullet_ja)\right) \\
						&\hspace{17em}\pmod{O_{m+n-1}}.
					\end{align*}
					By \Cref{grading-2}, each term on the right-hand side lies in $O_{m+n}$, which completes the proof.
			\end{proof}\end{proposition}
			
			We have now established analogues of 
				\eqref{commutator}, \eqref{associator}, \eqref{translation}, \eqref{generating}, \eqref{derivation} and \eqref{shifting}. 
			Finally, we state a corollary from \Cref{grading-1}. \Cref{commutator-1} supplements the result in \Cref{Jacobi}(i). 
			It will be used in the next section.
			
			\begin{corollary}\label{commutator-1}
				When $c\in O_n$, for $a,b$ homogeneous we have
				\begin{align*}
					a*(b*c)-b*(a*c)\equiv\sum_{j=0}^\infty\binom{-1}j&(a_{(j)}b)\bullet_{-2-j}c \\
					&\pmod{O_{\deg a+\deg b+n-1}}.
				\end{align*}
			\begin{proof}
				By \Cref{grading-1}, we have $V_{<\deg a+\deg b}*c\subseteq O_{\deg a+\deg b+n-1}$. Now the result follows directly from \eqref{commutator-1-eq}.
			\end{proof}\end{corollary}
			
	\section{Proof of the main theorem}\label{sec4}
		In this section, $V$ denote an $\mathbb{N}$-graded vertex algebra.
		\subsection{Notations}
			In view of properties established in \Cref{sec3}, it is natural to associate a finite sum
			\begin{equation}\label{mode-product}
				\sum_jc_ja^{j,1}_{(-1-n_{j,1})}\cdots a^{j,s_j}_{(-1-n_{j,s_j})}\bm1
			\end{equation}
			with the corresponding one, by replacing mode products $_{(n)}$ with $\bullet_n$:
			\begin{equation}\label{Zhu-product}
				\sum_jc_ja^{j,1}\bullet_{-1-n_{j,1}}\cdots a^{j,s_j}\bullet_{-1-n_{j,s_j}}\bm1.
			\end{equation}
			However, since different expressions may be the same element in $V$, 
				the correspondence between \eqref{mode-product} and \eqref{Zhu-product} can only be formulated at the level of formal expressions, 
				rather than elements of $V$.
			We now introduce a formal algebra of expressions that allows us to describe this correspondence systematically.
			\begin{definition}\label{algebra of expressions}
				We define the set (without linear relations)
				$$S=\{a\mid a\in V\text{ is homogeneous}\}$$
				and as in \Cref{jet algebra}, let
				$$JS=\{\partial^na\mid a\in S\}$$
				Let $JS^*$ denote the free monoid generated by $JS$, and $\mathbb{C}[JS^*]$ be the associated monoid algebra. 
				The derivation $\partial$ acts on $\mathbb{C}[JS]$ and $\mathbb{C}[JS^*]$ naturally:
				$$\partial(\partial^na)=\partial^{n+1}a.$$
				We bi-grade $\mathbb{C}[JS]$ and $\mathbb{C}[JS^*]$ by
				\begin{align*}
					\deg\partial^na&=\deg a+n, \\
					\dep\partial^na&=n,
				\end{align*}
				and let $\ldep$ be the minimal $\dep$-degree.
				
				Now, we define two $\mathbb{C}[JS^*]$-actions on $V$, by
				\begin{equation}\label{mode-action}
					(\partial^na)b=(T^na)_{-1}b=n!a_{(-1-n)}b
				\end{equation}
				and
				\begin{equation}\label{Zhu-action}
					(\partial^na)b=n!a\bullet_{-1-n}b
				\end{equation}
				respectively.
				For each expression $P\in\mathbb{C}[JS^*]$, we denote the action of $P$ induced from \eqref{mode-action} by $\overline{P}$, 
					and the action of $P$ induced from \eqref{Zhu-action} by $\widetilde{P}$.
				
				Finally, note that the polynomial ring $\mathbb{C}[JS]$ is exactly the abelianization of $\mathbb{C}[JS^*]$, we let
				\begin{align*}
					\langle-\rangle:\mathbb{C}[JS^*]&\twoheadrightarrow\mathbb{C}[JS] \\
					P&\mapsto\langle P\rangle
				\end{align*}
				be the natural projection.
			\end{definition}
			
			In our construction, $S$ is a large strong generating set (all homogenous elements) for $V$, 
				$\mathbb{C}[JS]$ is the jet algebra of polynomial ring $\mathbb{C}[S]$, 
				and $\mathbb{C}[JS^*]$ describes expressions generated by $\{a_{(-1-n)}\}$ or $\{a\bullet_{-1-n}\}$. 
			Indeed, now \eqref{mode-product} and \eqref{Zhu-product} are exactly $\overline{P}\bm1$ and $\widetilde{P}\bm1$, 
				which are implicitly related via expression $P$.
			
			The following two lemmas are translated from \Cref{expansion}, \Cref{grading-2}, \Cref{Zhu-derivation} and \Cref{grading-1}.
			\begin{lemma}\label{restatement of formulae}
				Let $P\in\mathbb{C}[JS^*]$ be $\deg$-homogeneous, and $b\in V_{\leq n}$. Then
				\begin{enumerate}[label=(\roman*)]
					\item	$\widetilde{P}b\equiv\overline{P}b\pmod{V_{<\deg P+n}}$,
					\item	If $b\in O_n$, then $\widetilde{P}b\in O_{\deg P+n}$,
					\item	If $\ldep P>0$, then $\widetilde{P}b\in O_{\deg P+n}$.
				\end{enumerate}
			\begin{proof}
				Without loss of generality, we assume that $P$ is a monomial:
				$$P=(\partial^{n_1}a^1)(\partial^{n_2}a^2)\cdots(\partial^{n_s}a^s)$$
				and argue by induction on $s$. Let $P_1=\partial^{n_1}a^1$ and $P=P_1P'$.
				\begin{enumerate}[label=(\roman*)]
					\item	The case $s=1$ is exactly \Cref{expansion}. If $s>1$, then
						$$\widetilde{P'}b\equiv\overline{P'}b\pmod{V_{<\deg P'+n}}$$
						follows from the induction hypothesis, and hence
						$$\widetilde{P}b=\widetilde{P_1}\widetilde{P'}b\equiv \overline{P_1}\widetilde{P'}b
							\equiv\overline{P_1}\,\overline{P'}b=\overline{P}b\pmod{V_{<\deg P+n}}.$$
					\item	The case $s=1$ is exactly \Cref{grading-1} (for $n_1=0$) and \Cref{grading-2} (for $n_1>0$). If $s>1$, then
						$$\widetilde{P'}b\in O_{\deg P'+n}$$
						follows from the induction hypothesis, and hence
						$$\widetilde{P}b\in\widetilde{P_1}O_{\deg P'+n}\subseteq O_{\deg P+n}.$$
					\item	If $n_1>0$ (this includes the case $s=1$), then by \Cref{grading-2},
						$$\widetilde{P}b\in n_1!a^1\bullet_{-1-n_1}V_{\leq\deg P'+n}\in O_{\deg P+n}.$$
						Otherwise $s>0$ and $\dep P'>0$, hence
						$$\widetilde{P'}b\in O_{\deg P'+n}$$
						by the induction hypothesis. Now applying \Cref{grading-1} yields
						\begin{equation*}
							\widetilde{P}b\in a^1*O_{\deg P'+n}\subseteq O_{\deg P+n}.
						\qedhere\end{equation*}\end{enumerate}\end{proof}\end{lemma}
			
			\begin{lemma}\label{restatement of derivation}
				Let $P\in\mathbb{C}[JS^*]$ be $\deg$-homogeneous. Then
				\begin{enumerate}[label=(\roman*)]
					\item	$T\overline{P}=\overline{\partial P}+\overline{P}T$,
					\item	$\widetilde{T}\widetilde{P}b\equiv\widetilde{\partial P}b+\widetilde{P}\widetilde{T}b\pmod{O_{\deg P+n}}$ for $b\in V_{\leq n}$.\\
						In particular, we have $\widetilde{T}\widetilde{P}\bm1\equiv\widetilde{\partial P}\bm1\pmod{O_{\deg P}}$.
				\end{enumerate}
			\begin{proof}
				Without loss of generality, we assume that $P$ is a monomial:
				$$P=(\partial^{n_1}a^1)(\partial^{n_2}a^2)\cdots(\partial^{n_s}a^s)$$
				and argue by induction on $s$. Let $P_1=\partial^{n_1}a^1$ and $P=P_1P'$.
				\begin{enumerate}[label=(\roman*)]
					\item	The case $s=1$ follows immediately from \eqref{covariance}:
						$$[T,\overline{P}]=n_1![T,a^1_{(-1-n_1)}]=(n_1+1)!a^1_{(-2-n_1)}=\overline{\partial P}.$$
						If $s>1$, by applying the induction hypothesis, we obtain
						$$[T,\overline{P}]=\overline{\partial P_1}\,\overline{P'}+\overline{P_1}\,\overline{\partial P'}=\overline{\partial P}.$$
					\item	If $s=1$, then by \Cref{Zhu-derivation},
						\begin{align*}
							[\widetilde{T},\widetilde{P}]b&=n_1!(\widetilde{T}a^1)\bullet_{-1-n_1}b \\
							&\equiv(\widetilde{T}^{n_1}\widetilde{T}a^1)*b&\pmod{O_{\deg P+n}} \\
							&\equiv(n_1+1)!a^1\bullet_{-2-n_1}b&\pmod{O_{\deg P+n}} \\
							&=\widetilde{\partial P}b.
						\end{align*}
						If $s>1$, then the induction hypothesis implies that
						\begin{equation}\label{2nd part}
							[\widetilde{T},\widetilde{P'}]b\equiv\widetilde{\partial P'}b\pmod{O_{\deg P'+n}}.
						\end{equation}
						Acting $\widetilde{P_1}$ on \eqref{2nd part} and applying \Cref{restatement of formulae}(ii), we obtain
						\begin{equation*}
							[\widetilde{T},\widetilde{P}]b\equiv\widetilde{\partial P_1}\widetilde{P'}b+\widetilde{P_1}\widetilde{\partial P'}b
								=\widetilde{\partial P}b\pmod{O_{\deg P+n}}.
						\qedhere\end{equation*}\end{enumerate}\end{proof}\end{lemma}
			
		\subsection{Comparison between \texorpdfstring{$\overline{P}$}{} and \texorpdfstring{$\widetilde{P}$}{}}
			\begin{lemma}\label{deepen}
				Let $P\in\mathbb{C}[JS^*]$ be an expression such that $\deg P=d,\dep P>0$ and $\left<P\right>=0$, 
					then there exists $P'\in\mathbb{C}[JS^*]$ satisfying the following properties:
				\begin{enumerate}[label=(\arabic*)]\label{condition list}
					\item	$\overline{P}\bm1=\overline{P'}\bm1$,
					\item	$\widetilde{P}\bm1\equiv\widetilde{P'}\bm1\pmod{O_{d-1}}$,
					\item	$\deg P'=d$,
					\item	$\ldep P'>\dep P$.
				\end{enumerate}
				\begin{proof}
					$P$ lies in the commutator ideal of $\mathbb{C}[JS^*]$ since $\left<P\right>=0$. 
					Hence it suffices to consider $P$ of the following form, where $P_1,P_2$ are monomials:
					$$P=P_1\left[\partial^ma,\partial^nb\right]P_2.$$
					By Borcherds identity \eqref{commutator}, we have
					\begin{align}
					\notag
						\overline{P}&=m!n!\overline{P_1}[a_{(-1-m)},b_{(-1-n)}]\overline{P_2} \\
					\label{mode-commutator}
						&=m!n!\overline{P_1}\sum_{j=0}^\infty\binom{-1-m}{j}(a_{(j)}b)_{(-2-m-n-j)}\overline{P_2}.
					\end{align}
					On the other hand, since
					$$0<\dep P=\dep P_1+m+n+\dep P_2,$$ 
					at least one of the following conditions holds:
					\begin{itemize}
						\item	If $\dep P_1>0$: By \Cref{restatement of formulae}(i), we have
							\begin{align}
							\notag
								&\widetilde{[\partial^ma,\partial^nb]}\widetilde{P_2}\bm1 \\
							\notag
								\equiv&\overline{[\partial^ma,\partial^nb]P_2}\bm1&\pmod{V_{<d-\deg P_1}} \\
							\notag
								=&m!n!\sum_{j=0}^\infty\binom{-1-m}{j}(a_{(j)}b)_{(-2-m-n-j)}\overline{P_2}\bm1 \\
							\label{apply Borcherds}
								\equiv&m!n!\sum_{j=0}^\infty\binom{-1-m}{j}(a_{(j)}b)\bullet_{-2-m-n-j}\widetilde{P_2}\bm1&\pmod{V_{<d-\deg P_1}}.
							\end{align}
							According to \Cref{restatement of formulae}(iii), acting $\widetilde{P_1}$ on \eqref{apply Borcherds} yields
							\begin{equation}\label{Zhu-commutator}
								\widetilde{P}\bm1
									\equiv m!n!\widetilde{P_1}\sum_{j=0}^\infty\binom{-1-m}{j}(a_{(j)}b)\bullet_{-2-m-n-j}\widetilde{P_2}\bm1\pmod{O_{d-1}}.
							\end{equation}
						\item	If $m+n>0$, then applying \eqref{commutator-eq} gives
							\begin{equation}\label{absorb}\begin{aligned}
								\widetilde{[\partial^ma,\partial^nb]}\widetilde{P_2}\bm1
									\equiv m!n!\sum_{j=0}^\infty\binom{-1-m}{j}(a_{(j)}b)&\bullet_{-2-m-n-j}\widetilde{P_2}\bm1 \\
								&\pmod{O_{d-\deg P_1-1}}.
							\end{aligned}\end{equation}
						\item	If $\dep P_2>0$, then $\widetilde{P_2}\bm1\in O_{\deg P_2}$ due to \Cref{restatement of formulae}(iii). 
							Then by \Cref{grading-1} we obtain \eqref{absorb} again.
					\end{itemize}
					According to \Cref{restatement of formulae}(ii), 
						by acting $\widetilde{P_1}$ on \eqref{absorb}, we obtain \eqref{Zhu-commutator} in the last two cases. 
					Note that $a_{(j)}b=0$ for $j\geq\deg a+\deg b$, 
						hence by comparing \eqref{mode-commutator} and \eqref{Zhu-commutator}, we see the expression
					$$P'=P_1\sum_{j=0}^{\deg a+\deg b}\sum_{a_{(j)}b\neq0}\binom{-1-m}{j}\frac{m!n!}{(m+n+j+1)!}(\partial^{m+n+j+1}(a_{j}b))P_2$$
					satisfies conditions \hyperref[condition list]{(1)(2)}. 
					Condition \hyperref[condition list]{(3)} holds because Borcherds identity \eqref{commutator} is homogeneous. 
					Finally, condition \hyperref[condition list]{(4)} is easy to check:
					\begin{equation*}
						\ldep P'\geq\dep P_1+m+n+1+\dep P_2=\dep P+1.
					\qedhere\end{equation*}\end{proof}\end{lemma}
			
			\begin{lemma}\label{classical freeness restatement}
				Let $V$ be classically free. If an expression $P\in\mathbb{C}[JS^*]$ satisfies
				
				\begin{enumerate*}[label=(\arabic*), itemjoin=\qquad]
					\item	$\dep P=l$, and
					\item	$\overline{P}\bm1\in F^{l+1}V$,
				\end{enumerate*}\\
				then there exists a finite family of $n_i\geq0$ and $Q_i,G_i\in\mathbb{C}[JS^*]$, with $\dep G_i=0$ and $\overline{G_i}\bm1\in F^1V$, such that
				\begin{equation}\label{generator}
					\langle P\rangle=\sum_i\langle Q_i\rangle\partial^{n_i}\langle G_i\rangle.
				\end{equation}
				\begin{proof}
					Consider the map
					\begin{align*}
						\varphi:\mathbb{C}[JS]&\to\gr V \\
						\langle P\rangle&\mapsto\overline{P}\bm1+F^{l+1}V\in(\gr V)_l,\quad l=\dep P.
					\end{align*}
					By \eqref{mode-action}, $\varphi$ is exactly the natural projection \eqref{natural projection}.
					
					According to the definition, we have
					\begin{equation}\label{ker from map}
						(\ker\varphi)_{\dep=l}=\{\langle P\rangle|\dep P=l,\overline{P}\bm1\in F^{l+1}V\}.
					\end{equation}
					Let $I=\ker\varphi|_{\mathbb{C}[S]}$ and $\mathcal{I}\subseteq\mathbb{C}[JS]$ be the differential ideal generated by $I$. 
					Then we have the induced homomorphism
					$$\overline{\varphi}:\mathbb{C}[JS]/\mathcal{I}\to\gr V=\mathcal{J}_\infty R_V,$$
					which is an isomorphism at $\dep=0$ component:
					$$\overline{\varphi}_{\dep=0}:\mathbb{C}[S]/I\xrightarrow{\sim}R_V.$$
					Meanwhile, according to the construction in \Cref{jet algebra}, 
					$\mathbb{C}[JS]/\mathcal{I}$ is exactly the jet algebra $\mathcal{J}_\infty(\mathbb{C}[S]/I)$.
					Hence $\overline{\varphi}$ is an isomorphism by universal property of jet algebra. It follows that $\ker\varphi=\mathcal{I}$, and thus
					\begin{equation}\label{ker from generator}
						(\ker\varphi)_{\dep=l}=\spn\{\langle Q\rangle\partial^n\langle G\rangle|\dep Q+n=l,\langle G\rangle\in I\}.
					\end{equation}
					Now comparing \eqref{ker from map} and \eqref{ker from generator} completes the proof.
			\end{proof}\end{lemma}
			
			\begin{proposition}\label{leading}
				Let $V$ be classically free. 
				If an expression $P\in\mathbb{C}[JS^*]$ satisfies $\deg P=d,\ldep P>0$ and $\overline{P}\bm1=0$, then $\widetilde{P}\bm1\in O_{d-1}$.
				\begin{proof}
					Fix $d$. We proceed by decreasing induction on $l=\ldep P$.
					When $l$ is sufficiently large (e.g. $l>d$), there is no such $P$, and thus the proposition holds trivially. 
					For a general $l$, we consider the $\dep=l$ component of $P$. Let
					\begin{equation}\label{bottom term}
						P=P_l+P_+,\quad\text{where }\dep P_l=l\text{ and }\ldep P_+>l.
					\end{equation}
					Then $\overline{P}\bm1=0$ implies that $\overline{P_l}\bm1\in F^{l+1}V$. By \Cref{classical freeness restatement}, we have
					\begin{equation}\label{element of differential ideal}
						\langle P_l\rangle=\sum_i\langle Q_i\rangle\partial^{n_i}\langle G_i\rangle,\quad
							\text{where }\dep G_i=0\text{ and }\overline{G_i}\bm1\in F^1V.
					\end{equation}
					
					First, we rearrange the sum \eqref{element of differential ideal}. 
					Since $F^1V\subseteq V$ is a homogeneous subspace, 
						we can assume that $\deg G_i$ is well-defined (by necessarily splitting $G_i$ into its $\deg$-homogenous components). 
					Since $\dep P_l=l$ and $\dep\partial^{n_i}G_i=n_i$ are well-defined, we may replace each $Q_i$ with its $\dep=l-n_i$ component. 
					Thus we can assume that $\dep Q_i$ is also well-defined. Hence, we can assume
					\begin{align*}
						d=\deg P_l&=\deg Q_i+n_i+\deg G_i, \\
						l=\dep P_l&=\dep Q_i+n_i+\dep G_i
					\end{align*}
					for all $i$'s, without loss of generality.
					
					Since $\overline{G_i}\bm1\in F^1V$, 
						we can choose $G_i'$ with $\deg G_i=\deg G_i'$ and $\ldep G_i'>0$, such that $\overline{G'_i}\bm1=\overline{G_i}\bm1$. 
					Let
					\begin{align*}
						P_0&=\sum_iQ_i\partial^{n_i}G_i, \\
						P'&=\sum_iQ_i\partial^{n_i}G_i'.
					\end{align*}
					By \Cref{restatement of derivation}(ii), we have
					\begin{align*}
						\widetilde{P_0}\bm1=\sum_i\widetilde{Q_i}\widetilde{\partial^{n_i}G_i}\bm1
						&\equiv\sum_i\widetilde{Q_i}\widetilde{T}^{n_i}\widetilde{G_i}\bm1\pmod{O_{d-1}}, \\
						\widetilde{P'}\bm1=\sum_i\widetilde{Q_i}\widetilde{\partial^{n_i}G_i'}\bm1
						&\equiv\sum_i\widetilde{Q_i}\widetilde{T}^{n_i}\widetilde{G_i'}\bm1\pmod{O_{d-1}}.
					\end{align*}
					And by construction of $G_i'$, we have
					$$\widetilde{Q_i}\widetilde{T}^{n_i}(\widetilde{G_i}\bm1-\widetilde{G_i'}\bm1)\in\widetilde{Q_i}\widetilde{T}^{n_i}V_{<\deg G_i}.$$
					Notice that $\dep Q_i+n_i=l>0$ implies that either $\dep Q_i>0$ or $n_i>0$.
					
					If $\dep Q_i>0$, then by \Cref{restatement of formulae}(iii), we have
					$$\widetilde{Q_i}\widetilde{T}^{n_i}V_{<\deg G_i}\subseteq\widetilde{Q_i}V_{<n_i+\deg G_i}\subseteq O_{d-1}.$$
					And if $n_i>0$, then by \Cref{restatement of formulae}(ii), we also have
					$$\widetilde{Q_i}\widetilde{T}^{n_i}V_{<\deg G_i}\subseteq\widetilde{Q_i}O_{n_i+\deg G_i-1}\subseteq O_{d-1}.$$
					Here we use the obvious fact that $\widetilde{T}V_{\leq n}=V_{\leq n}\circ\bm1\subseteq O_{n+1}$.
					
					Hence we have
					\begin{equation}\label{transition}
						\widetilde{P_0}\bm1\equiv\widetilde{P'}\bm1\pmod{O_{d-1}}.
					\end{equation}
					On the other hand, note that $\langle P_l\rangle=\langle P_0\rangle$, we apply \Cref{deepen} to expression $P_l-P_0$. 
					Then we obtain an expression $R$ with $\ldep R>l$, such that
					\begin{align}
					\label{mode-rewritten}
						\overline{R}\bm1&=\overline{P_l}\bm1-\overline{P_0}\bm1, \\
					\label{Zhu-rewritten}
						\widetilde{R}\bm1&\equiv\widetilde{P_l}\bm1-\widetilde{P_0}\bm1\pmod{O_{d-1}}.
					\end{align}
					Note that the construction of $G_i'$ guarantees that $\overline{P_0}\bm1=\overline{P'}\bm1$. 
					Now it follows from \eqref{bottom term}, \eqref{transition}, \eqref{mode-rewritten} and \eqref{Zhu-rewritten} that
					\begin{align*}
						0=\overline{P}\bm1&=\overline{P'}\bm1+\overline{R}\bm1+\overline{P_+}\bm1, \\
						\widetilde{P}\bm1&\equiv\widetilde{P'}\bm1+\widetilde{R}\bm1+\widetilde{P_+}\bm1\pmod{O_{d-1}}.
					\end{align*}
					Since $\ldep(P'+R+P_+)>l$, the induction hypothesis applies.
			\end{proof}\end{proposition}
			
			The use of formal expressions is now ended. Following we translate the result of \Cref{leading} into usual notations.
			\begin{proposition}\label{restate of leading}
				If $V$ is classically free, then for all $d\geq0$ we have
				\begin{equation}\label{reduced state}
					O(V)\cap V_{\leq d}=O_d.
				\end{equation}
			\begin{proof}
				First we prove that
				\begin{equation}\label{induction state}
					O_d\cap V_{<d}\subseteq O_{d-1}
				\end{equation}
				If $x\in O_d\cap V_{<d}$, then we may write
				\begin{equation}\label{typical element in Od}
					x\equiv\sum_ia^i\circ b^i\pmod{O_{d-1}},
				\end{equation}
				where $\deg a^i+\deg b^i+1=d$. Let
				$$P=\sum_i(\partial a^i)b^i.$$
				Then $\deg P=d,\dep P=1>0$, and $x\in V_{<d}$ means that $\overline{P}\bm{1}=0$. 
				Hence by \Cref{leading} and \eqref{typical element in Od}, we have
				$$x\equiv\widetilde{P}\bm1\equiv0\pmod{O_{d-1}}.$$
				
				Now if $x\in O(V)\cap V_{\leq d}$, then we can assume that $x\in O_n$ for some $n>d$. 
				Applying \eqref{induction} iterately, we obtain
				$$x\in O_n\cap V_{\leq d}=O_n\cap V_{\leq n-1}\cap\cdots\cap V_{\leq d}\subseteq O_d$$
				Since $O_{d-1}\subseteq O(V)\cap V_{\leq d}$ is evident, the proof is complete.
		\end{proof}\end{proposition}
	
		\subsection{Proof of Main Theorem}
			We now prove that $R_V\cong\gr A(V)$ when $V$ is classically free.
			\begin{lemma}\label{induction}
				The natural projection $\eta_V:R_V\twoheadrightarrow\gr A(V)$ is an isomorphism, if and only if \eqref{reduced state} holds for all $d\geq0$.
				\begin{proof}
					Suppose that $R_V\cong\gr A(V)$, and $x\in O(V)\cap V_{\leq d}$. We show that $x\in O_d$ by induction on $d$. Let
					$$x=v+\text{lower degree components}.$$
					Then $v+F^1V\in\ker\eta_V$, and hence $v\in F^1V$.
					
					If $d=0$, then $v\in V_0\cap F^1V=0$, and hence $x=0\in O_0$.
					
					If $d>0$, we can write
					$$v=\sum_ja^j_{(-2)}b^j,\quad\deg v=\deg a^j+\deg b^j+1.$$
					Then the induction hypothesis implies that
					$$x-\sum_ja^j\circ b^j\in O(V)\cap V_{\leq d-1}=O_{d-1}.$$
					Note that $a^j\circ b^j\in O_d$. Therefore, $x\in O_d+O_{d-1}=O_d$.
					
					Conversely, suppose that \eqref{reduced state} holds for all $d\geq0$. Let $v\in V_d$ such that $v+F^1V\in\ker\eta_V$. Then we have
					\begin{equation}\label{representative element}
						v\equiv\sum_ia^i\circ b^i\pmod{V_{d-1}}.
					\end{equation}
					We argue by induction on
					$$s=\max_i(\deg a^i+\deg b^i+1)-d\geq0.$$
					Partition the terms according to degree:
					\begin{align*}
						I_1&=\{i\mid\deg a^i+\deg b^i+1=s+d\}, \\
						I_2&=\{i\mid\deg a^i+\deg b^i+1<s+d\}.
					\end{align*}
					When $s=0$, we have
					$$v=\sum_{i\in I_1}a^i_{(-2)}b^i\in F^1V.$$
					Assume now that $s>0$. Then the $\deg=s+d$ component of \eqref{representative element} vanishes:
					$$\sum_{i\in I_1}a^i_{(-2)}b^i=0.$$
					By \eqref{reduced state}, we obtain
					$$\sum_{i\in I_1}a^i\circ b^i\in O(V)\cap V_{\leq s+d-1}=O_{s+d-1}.$$
					Note that $\sum\limits_{i\in I_2}a^i\circ b^i$ automatically lies in $O_{s+d-1}$, hence
					$$\sum_ia^i\circ b^i\in O_{s+d-1}.$$
					That is, there exist $x^j,y^j$ with $\deg x^j+\deg y^j+1<s+d$, such that
					$$v\equiv\sum_ia^i\circ b^i=\sum_jx^j\circ y^j\pmod{V_{d-1}}.$$
					$\max\limits_j(\deg x^j+\deg y^j+1)-d<s$, thus the induction hypothesis applies.
			\end{proof}\end{lemma}
			
			\begin{theorem}\label{main}
				If $V$ is classically free, then the natural projection
				$$\eta_V:R_V\twoheadrightarrow\gr A(V)$$
				defined in \Cref{epimorphism} is an isomorphism.
				\begin{proof}
					It follows immediately from \Cref{restate of leading} and \Cref{induction}.
			\end{proof}\end{theorem}
			
	\section{Applications of the main theorem}\label{sec5}
		Since $R_V$ and $ A(V)$ are usually easier to compute than $\gr V$ for general $V$, we disprove the classical freeness for many VOAs.
		\begin{theorem}\label{Zhu=C}
			Let $V$ be a non-trivial holomorphic VOA of CFT type. Then $V$ is not classically free.
		\end{theorem}
		\begin{proof}
			By \Cref{Zhu theory}(iv), $ A(V)$ is a finite dimensional semisimple algebra whose unique simple module is $V_0\cong\mathbb{C}$.
			Hence by Artin-Wedderburn theorem, we have
			$$ A(V)\cong\mathbb{C}.$$
			Assume that $V$ is classically free, then by \Cref{main} we have
			\begin{equation}\label{dimR_V}
				R_V\cong\gr A(V)\cong\mathbb{C},
			\end{equation}
			and thus
			$$\gr V\cong\mathcal{J}_\infty\mathbb{C}=\mathbb{C}.$$
			Therefore, we have $(\gr V)_l=0$ for $l>0$, i.e.
			\begin{equation}\label{degeneration}
				F^1V=F^2V=\cdots=F^lV=\cdots.
			\end{equation}
			On the other hand, for an $\mathbb{N}$-graded VOA we have $F^pV\subseteq V_{\geq p}$, 
			which implies that Li filtration is separated, i.e.
			\begin{equation}\label{separated}
				\bigcap_{p>0}F^pV=0
			\end{equation}
			Now it follows from \eqref{degeneration}, \eqref{separated} and \eqref{dimR_V} that
			$$V=V/\bigcap_{p>0}F^pV=V/F^1V=R_V=\mathbb{C},$$
			contradicting with that $V$ is non-trivial. Hence $V$ is not classically free.
		\end{proof}
		
		\begin{corollary}
			Let $L$ be a non-trivial positive-definite even self-dual lattice, then the lattice VOA $V_L$ is not classically free.
			\begin{proof}
				Since $L$ is positive-definite even, $V_L$ is rational \cite{B, FLM, D2}. 
				According to \cite[Theorem 3.1]{D2}, all irreducible admissible $V_L$-modules are
				$$V_{\lambda+L},\quad\lambda\in L^*/L.$$
				Hence $V_L$ is holomorphic if $L^*=L$. Note that $V_L$ is automatically of CFT type. Therefore, $V_L$ is not classically free due to \Cref{Zhu=C}.
		\end{proof}\end{corollary}
			
		Among simple root lattices, only the $E_8$ root lattice is even and self-dual.
		\begin{example}\label{E8}
			$L_1(E_8)=V_{E_8}$ is not classically free.
		\end{example}
		
		Another famous example of positive-definite even self-dual lattice is the Leech lattice.
		\begin{example}\label{Leech}
			Let $\Lambda$ be the Leech lattice, then $V_\Lambda$ is not classically free.
		\end{example}
		
		Moreover, the  Moonshine VOA is rational \cite{B, FLM, D1}, holomorphic \cite{D1}, and of CFT type since it is an orbifold of $V_\Lambda$.
		\begin{example}\label{moonshine}
			The  Moonshine VOA is not classically free.
		\end{example}
		
	\section{Direct calculation of Lower weights of \texorpdfstring{$\gr L_1(E_8)$}{}}\label{sec6}
		In the last section, we provide a direct calculation proving that $L_1(E_8)$ is not classically free. 
		Let $\mathfrak{g}$ be the Lie algebra of $E_8$, and let $\mathfrak{h}$ be its Cartan subalgebra.
		\subsection{Character of affine vertex algebra of \texorpdfstring{$E_8$}{} at level one}\mbox{}
			
			Consider the $q$-dimension of the lattice VOA $V_{E_8}=L_1(E_8)$:
			$$\dim_qV_{E_8}=\Tr_{V_{E_8}}q^{L_0}.$$
			Since $V_{E_8}=V_{\mathfrak{h}}\otimes\mathbb{C}[Q_{E_8}]$, where $Q(E_8)$ is the root lattice of $E_8$, we have
			\begin{equation}\label{q-dim}
				\dim_qV_{E_8}=\Tr_{V_\mathfrak{h}}q^{L_0}\cdot\Tr_{\mathbb{C}[Q_{E_8}]}q^{L_0}=\frac{\Theta_{E_8}(\tau)}{\varphi(q)^8},
			\end{equation}
			where
			$$\varphi(q)=\prod_{n=1}^\infty(1-q^n)$$
			is the Euler product, and
			$$\Theta_{E_8}(\tau)=\sum_{\alpha\in Q(E_8)} q^{(\alpha,\alpha)/2},\qquad q=e^{2\pi i\tau}.$$
			We can calculate that
			\begin{align}
				\notag
				\varphi(q)^{-8}&=(1-8q+20q^2-70q^4+\cdots)^{-1} \\
				\label{Theta}
				&=1+8q+44q^2+192q^3+726q^4+\cdots.
			\end{align}
			Using identity $\Theta_{E_8}(\tau)=E_4(\tau)$ (see \cite[Chapter VII, \S6.6]{S}), we obtain
			\begin{equation}\label{Eisenstein}
				\Theta_{E_8}(\tau)=E_4(\tau)=1+240q+2160q^2+6720q^3+17520q^4+\cdots.
			\end{equation}
			Hence by \eqref{q-dim}, \eqref{Theta} and \eqref{Eisenstein} we have
			\begin{equation}\label{q-expansion}
				\dim_qV_{E_8}=1+248q+4124q^2+34752q^3+213126q^4+\cdots.
			\end{equation}
			
		\subsection{Tensor product of irreducible modules of \texorpdfstring{$E_8$}{}}\mbox{}
			
			In this subsection we list dimensions of irreducible highest weight $\mathfrak{g}$-modules involved, as well as some decompositions we need. 
			All the results are computed by Sage.
			
			Let $w_i,1\leq i\leq 8$ be the fundamental weights of $E_8$. 
			The following are the dimensions of some irreducible highest weight $\mathfrak{g}$-modules:
			
			\begin{center}\begin{tabular}{rccccc}\hline
				weights: & $w_1$ & $w_2$ & $w_3$ & $w_4$ &\\
				dimension:	& 3875 & 147250 & 6696000 & 6899079264 &\\ \hline
				weights: & $w_5$ & $w_6 $& $w_7$ & $w_8$ & $2w_8$\\
				dimension:	& 146325270 & 2450240 & 30380 & 248 & 27000 \\ \hline
			\end{tabular}\end{center}
			
			Let $L(\lambda)$ be the irreducible $\mathfrak{g}$-module with highest weight $\lambda$. Then
			$$L(w_8)\cong\mathfrak{g}.$$
			Denote $L(2w_8)$ by $I$, then we have the following decompositions:
			\begin{align}
			\label{S2}
				S^2(\mathfrak{g})&\cong\mathbb{C}\oplus L(w_1)\oplus I, \\
			\label{g*w1}
				\mathfrak{g}\otimes L(w_1)&\cong L(w_7)\oplus\mathfrak{g}\oplus L(w_1+w_8)\oplus L(w_1)\oplus L(w_2), \\
			\label{g*I}
				\mathfrak g\otimes I&\cong L(w_7)\oplus L(w_7+w_8)\oplus\mathfrak{g}\oplus L(w_1+w_8)\oplus I\oplus L(3w_8), \\
			\label{S3}
				S^3(\mathfrak{g})&\cong L(w_7)\oplus\mathfrak{g}\oplus L(w_1+w_8)\oplus L(3w_8).
			\end{align}
			
		\subsection{Dimension contradiction}\mbox{}
			
			Let $x_i$ be a basis of $\mathfrak{g}$, $e_\theta$ be a vector of the highest root in $\mathfrak{g}$. Let
			$$R_V=S(\mathfrak{g})/J.$$
			Then $J$ admits a weight decomposition
			$$J=\bigoplus_{n\geq2}J_n,$$
			and $I\subseteq J_2$ (see \cite{K1}). 
			Assume that $L_1(\mathfrak{g})$ is classically free, then
			$$\gr(L_1(\mathfrak{g}))\cong\mathbb{C}[\partial^kx_i|k\in\mathbb{N}]/\langle\partial^ts|s\in J, t\in\mathbb{N}\rangle.$$
			Now we calculate the dimension of $\gr(L_1(\mathfrak{g}))=\bigoplus\limits_{n\in\mathbb{N}} V_n$ at each conformal weight $n$, 
				under the assumption that $L_1(E_8)$ is classically free. 
			Following we always decompose the spaces according to differential degree at first.
			
			\begin{enumerate}[label=(\arabic*), start=0]
				\item	$V_0=\mathbb{C}$, and $\dim V_0=1$.	
				\item 	$V_1=\mathfrak{g}$, and $\dim V_1=248$.
				\item 	By \eqref{S2} we have
					$$V_2=\frac{S^2(\mathfrak{g})}{J_2}\oplus\partial\mathfrak{g}\cong\frac{\mathbb{C}\oplus L(w_1)}{J_2/I}\oplus\mathfrak{g}.$$
					Taking the dimension yields
					$$\dim J_2/I=1+3875+248-4124=0.$$
					Hence $J_2=I$.
				\item	The ideal in $V_3$ is $\mathfrak{g}\otimes I\oplus\partial I$, therefore,
					\begin{align*}
						V_3&=\frac{S^3(\mathfrak{g})}{J_3}\oplus\frac{\mathfrak{g}\otimes\partial\mathfrak{g}}{\partial I}\oplus\partial^2\mathfrak{g} \\
						&\cong\frac{S^3(\mathfrak{g})}{J_3}\oplus\frac{\mathfrak{g}\otimes\mathfrak{g}}I\oplus\mathfrak{g}.
					\end{align*}
					Taking the dimension yields
					$$\dim S^3(\mathfrak{g})/J_3=34752-(248^2-27000)-248=0.$$
					Hence $J_3=S^3(\mathfrak{g})$, and it follows that
					\begin{equation}\label{ideal}
						J_n=S^n(\mathfrak{g}),\qquad n\geq3.
					\end{equation}
				\item	Using \eqref{ideal}, we obtain
					\begin{align*}
						V_4&=\frac{S^4(\mathfrak{g})}{J_4}
							\oplus\frac{S^2(\mathfrak{g})\otimes\partial\mathfrak{g}}{\mathfrak{g}\partial I+I\partial\mathfrak{g}+\partial J_3}
							\oplus\frac{\mathfrak{g}\otimes\partial^2\mathfrak{g}\oplus S^2(\partial\mathfrak{g})}{\partial^2I}\oplus\partial^3\mathfrak{g} \\
						&=\frac{S^2(\mathfrak{g})\otimes\partial\mathfrak{g}}{\mathfrak{g}\partial I+I\partial\mathfrak{g}+\partial S^3(\mathfrak{g})}
							\oplus\frac{\mathfrak{g}\otimes\mathfrak{g}\oplus S^2(\mathfrak{g})}I\oplus\mathfrak{g}.
					\end{align*}
					Taking dimension yields
					\begin{equation}\label{F1V4}
						\dim\frac{S^2(\mathfrak{g})\otimes\partial\mathfrak{g}}{\mathfrak{g}\partial I+I\partial\mathfrak{g}+\partial S^3(\mathfrak{g})}
							=213126-34752-\binom{249}2=147498.
					\end{equation}
					On the other hand, \eqref{S2} and \eqref{g*w1} show the inclusion
					$$L(w_1)\oplus L(w_2)\hookrightarrow L(w_1)\otimes\mathfrak{g}\hookrightarrow S^2(\mathfrak{g})\otimes\mathfrak{g}.$$
					Meanwhile, neither $\mathfrak g\otimes I$ nor $S^3(\mathfrak{g})$ contains a piece $L(w_1)$ or $L(w_2)$. 
					Let
					\begin{align*}
						\Sym_{12}:\mathfrak{g}\otimes\mathfrak{g}\otimes\partial\mathfrak{g}&\to S^2(\mathfrak{g})\otimes\partial\mathfrak{g} \\
						x\otimes y\otimes\partial z&\mapsto xy\otimes\partial z.
					\end{align*}
					Then
					\begin{align*}
						\mathfrak{g}\partial I&=\Sym_{12}(\mathfrak{g}\otimes\partial I), \\
						I\partial\mathfrak{g}&=\Sym_{12}(I\otimes\partial\mathfrak{g}).
					\end{align*}
					Note that
					$$\mathfrak{g}\otimes\partial I\cong\mathfrak{g}\otimes I\cong I\otimes\partial\mathfrak{g},$$
					therefore, both $\mathfrak{g}\partial I$ and $I\partial\mathfrak{g}$ are quotients of $g\otimes I$, 
						and hence do not contain a piece $L(w_1)$ or $L(w_2)$. 
					We now conclude that
					$$\dim\frac{S^2(\mathfrak{g})\otimes\partial\mathfrak{g}}{\mathfrak{g}\partial I+I\partial\mathfrak{g}+\partial S^3(\mathfrak{g})}
						\geq\dim L(w_1)+\dim L(w_2)=151125.$$
					This contradicts with \eqref{F1V4}. Hence $L_1(E_8)$ is not classically free.
			\end{enumerate}

\end{document}